\documentclass[12pt]{article}

\usepackage[a4paper,margin=1in]{geometry}
\usepackage{amsmath,amssymb,amsthm,amsfonts}
\usepackage{mathrsfs}
\usepackage{mathtools}
\usepackage{bm}
\usepackage{enumerate}
\usepackage{booktabs}
\usepackage{array}
\usepackage{multirow}
\usepackage{graphicx}
\usepackage{tikz}
\usepackage{caption}
\usepackage{subcaption}
\usepackage{hyperref,cite}
\usepackage{xcolor,float}
\usepackage{enumitem}
\usepackage{subcaption}
\usetikzlibrary{decorations.pathreplacing}
\usetikzlibrary{arrows.meta,positioning,calc,shapes.geometric,fit}

\hypersetup{
	colorlinks=true,
	linkcolor=blue,
	citecolor=blue,
	urlcolor=blue
}

\newtheorem{theorem}{Theorem}[section]

\newtheorem{proposition}[theorem]{Proposition}
\newtheorem{corollary}[theorem]{Corollary}
\theoremstyle{definition}
\newtheorem{definition}[theorem]{Definition}
\newtheorem{remark}[theorem]{Remark}

\newtheorem{conjecture}[theorem]{Conjecture}
\newtheorem{openproblem}[theorem]{Open Problem}

\renewenvironment{proof}{{\noindent\bfseries Proof.}}{\qed}
\title{On roots of domination polynomials for friendship and book graphs}
\author{
	{\small Bilal Ahmad Rather}\\[2mm]
	{\small School of Mathematics and Statistics, Shandong University of Technology, Zibo 255049, China}\\
	\texttt{bilahamadrr@gmail.com}
			}
\date{}
\begin{document}
	\maketitle
	
	\begin{abstract}
		This study examines the domination polynomials of friendship graphs and book graphs, focusing on unanswered questions related to these families [Alikhani,  Brown and  Jahari, 	on the domination polynomials of friendship graphs, Filomat \textbf{30}(1) (2016) 169--178]. For the friendship graph $F_n$, with even $n$, we show that the polynomial $D(F_n,x)$ has exactly three real zeros: $0$ and two simple zeros in the intervals $(-2,-1)$ and $(-1,0)$. We further show that these two nonzero zeros have monotonic variation and converge to $-1-\frac{1}{\sqrt2}$ and $-1+\frac{1}{\sqrt2}$, respectively. We obtain the quantitative approximation $(|z|-1)^2\log |z|\le n$ for any complex zeros of $D(F_n,x)$, resulting in the explicit bound $|z|\le 1+\sqrt{\tfrac{n}{\log 2}}$. For book graphs $B_n$, we ascertain the comprehensive limit set of domination roots and establish results about the presence of real roots contingent on parity. We provide a partial answer to the integer-root an issue by establishing that friendship and book graphs have no nonzero integer domination roots, whereas for corona families, the only nonzero integer root is $-2$.
	\end{abstract}
	
	
	\medskip
	
	
	\section{Introduction}
	
	Let $G=(V,E)$ be a finite simple graph. A subset $S\subseteq V$ is called a dominating set if every vertex of $V\setminus S$ has a neighbor in $S$. The domination polynomial of $G$ is the generating function of domination sets, and is defined as
	$$
	D(G,x)=\sum_{i=\gamma(G)}^{|V|} d(G,i)x^i,
	$$
	where $d(G,i)$ denotes the number of dominating sets of size $i$ and $\gamma(G)$ is the domination number of $G$. Domination theory is a well-established component of graph theory, driven by robust structural and algorithmic considerations, encompassing resource allocation, coverage, and network resilience issues; key references are \cite{HaynesHedetniemiSlater1998,AlikhaniPeng2014Intro}. From an algebraic perspective, domination polynomials encapsulate a significant amount of combinatorial information within a singular entity, and their zeros frequently reveals underlying regularity in the examined graph family. This philosophy is already recognized from independence, matching, and chromatic polynomials, where the positioning of zeros has resulted in both structural theorems and asymptotic phenomena \cite{BrownDilcherNowakowski2000,BrownNowakowski2001,BrownHickman2002,Comtet1974, BrownTufts2014}.
	
	The interest in dominance roots has increased due to the combinatorial richness and computational complexity of dominating sets. Generally, ascertaining whether $\gamma(G) \le k$ is NP-complete \cite{GareyJohnson1979}, hence precise formulations for $D(G,x)$ are seldom and much sought after. Numerous graph classes are recognized as being defined by their domination polynomial, although others are not; refer to \cite{AkbariAlikhaniPeng2010,AlikhaniPeng2011Cubic,AnthonyPicollelli2015,bilaldmlet, bilalarxiv}. The interest in dominance roots has increased significantly owing to the combinatorial intricacy and computing complexity of dominating sets. Determining whether $\gamma(G) \le k$ is NP-complete \cite{GareyJohnson1979}, hence exact formulations for $D(G,x)$ are rare and much desired. Several graph classes are identified by their domination polynomial, whereas others are not, see \cite{AkbariAlikhaniPeng2010,AlikhaniPeng2011Cubic,AnthonyPicollelli2015, bilalsc,bilaltcs,bilalac}. 
	
	Among the most natural triangle-rich families are the friendship graphs $ F_n=K_1+nK_2 $ and the $n$-book graphs $B_n$, obtained by identifying a common edge in $n$ copies of $C_4$. The friendship graphs are defined by the Erdős–Rényi–Sós friendship theorem \cite{ErdosRenyiSos1966}, establishing them as a standard test family for extremal and algebraic inquiries. In a seminal paper, Alikhani, Brown, and Jahari \cite{AlikhaniBrownJahari2016} formulated explicit equations for the domination polynomials of $F_n$ and $B_n$, examined their root configurations, demonstrated a limiting hyperbola for the roots of $F_n$, confirmed that $F_n$ is not $D$-unique for $n \ge 3$, and presented multiple unresolved questions regarding precise real-root counts, modulus constraints, book-graph root topology, and integer domination roots. The hypothesis presented in \cite{AlikhaniBrownJahari2016} was subsequently advanced in \cite{bilaljcmcc,Alikhani2013AtMinusOne,Alikhani2013NonP4}.
	
	This study addresses these unresolved inquiries. Specifically, we reexamine Questions 3.1, 3.2, 3.3, and 3.6 (see Open Problem \ref{problem}) from \cite{AlikhaniBrownJahari2016}. The initial principal theorem provides a definitive resolution to Question 3.1: when $n$ is even, $D(F_n,x)$ possesses precisely three real zeros. The demonstration is not a numerical argument; it relies on a rigorous monotonicity analysis of two logarithmic equations derived by moving the variable by $1$. We subsequently refine the asymptotic analysis by demonstrating the monotone convergence of the two nonzero real roots to the hyperbola intercepts already indicated in \cite{AlikhaniBrownJahari2016}. In Question 3.2, we formulate a novel quantitative assessment for all complex zeros of $D(F_n,x)$, specifically $ (|z|-1)^2\log |z|\le n,$ which immediately implies the explicit sublinear bound
	$|z|\le 1+\sqrt{\frac{n}{\log 2}}. $
	This provides a definitive, verified the upper bound on the modulus, significantly more precise than rudimentary coefficient bounds.
	
	Our examination of book graphs offers a limited yet thorough response to Question 3.3. Employing the results from \cite{BerahaKahaneWeiss1978,BrownHickman2002}, we ascertain the comprehensive limiting set of domination roots of $B_n$ as a union of three algebraic components along with two distinct points. We also establish parity-dependent existence assertions for the real roots of $D(B_n,x)$ and define the exact intersections with the real axis of the limit set. Ultimately, inspired by Question 3.6, we demonstrate that friendship graphs and book graphs lack nonzero integer domination roots. In conjunction with the corona constructions from \cite{AlikhaniBrownJahari2016}, this provides a partial affirmative response to the overarching integer-root problem: across the natural families examined, the only nonzero integer domination root that emerges is $-2$.
	
	The results are intresting in two specific respects. Initially, they resolve a previously unresolved precise enumeration issue concerning the real roots of friendship graphs. Secondly, they transform the qualitative root plots of \cite{AlikhaniBrownJahari2016} into quantitative theorems: precise root counts, monotonic convergence, a sublinear modulus bound, and limit characterization for the zeros of book graphs. The combination of analytic inequalities, recurrence-free algebraic transformations, and asymptotic root-limit techniques is particularly efficacious for domination polynomials of dense graph families exhibiting repetitive local structures.
	
	\medskip
	The article is organized as: Section~\ref{section 2} compiles the notation, fundamental identities, and established results utilized throughout the article. Section~\ref{section 3} provides a comprehensive resolution to Question 3.1 \cite{AlikhaniBrownJahari2016} by demonstrating that even friendship graphs possess precisely three real domination roots and by accurately identifying their two nonzero roots. Section~\ref{section 4} addresses Question 3.2 \cite{AlikhaniBrownJahari2016} by formulating a novel global modulus estimate for the complex roots of $D(F_n,x)$ and providing a clear sublinear upper bound. Section~\ref{section 5} addresses Question 3.3 \cite{AlikhaniBrownJahari2016}: we ascertain the  limit set for the roots of the book graphs and infer multiple implications regarding real roots. Section~\ref{section 6} pertains to Question 3.6 \cite{AlikhaniBrownJahari2016}. We show that friendship graphs and book graphs do not possess nonzero integer domination roots and clarify how this relates to known corona structures, identifying $-2$ as the only nonzero integer root in the previously understood families. Section~\ref{section 7} concludes the paper by presenting the conclusions, limitations, and possibilities for future research.
	
	\section{Preliminaries and notation}\label{section 2}
	
	We write $N(v)$ and $N[v]$ for the open and closed neighborhoods of a vertex $v$, and for a subset $S\subseteq V(G)$ we use $N[S]=\bigcup_{v\in S}N[v]$. A dominating set is a set $S$ with $N[S]=V(G)$. The set of distinct zeros of $D(G,x)$ is denoted by $Z(D(G,x))$.
	
	For $n\ge 1$, the friendship graph $F_n$ is the graph obtained by gluing $n$ triangles at a common vertex, see Figure \ref{fig:basicfamilies} (a). Equivalently,
	$ F_n=K_1+nK_2,$ where $+$ is join operation (obtained from two graphs by joining each vertex of first graph with every vertex of second graph).
	For $n\ge 1$, the book graph $B_n$ is formed by identifying a common edge in $n$ copies of $C_4$, see Figure \ref{fig:basicfamilies} (b).
	
	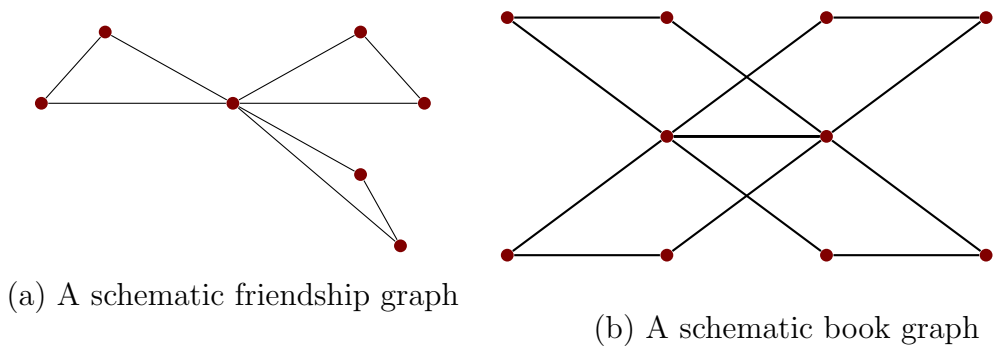
\begin{figure}[H]
		\centering
		\begin{tikzpicture}[scale=1.05, every node/.style={circle,fill=red!50!black,inner sep=1.7pt}]
			\node (c) at (0,0) {};
			\node (a1) at (1.6,0.9) {};
			\node (b1) at (2.4,0.0) {};
			\node (a2) at (1.6,-0.9) {};
			\node (b2) at (2.1,-1.8) {};
			\node (a3) at (-1.6,0.9) {};
			\node (b3) at (-2.4,0.0) {};
			\draw (c)--(a1)--(b1)--(c);
			\draw (c)--(a2)--(b2)--(c);
			\draw (c)--(a3)--(b3)--(c);
			\node[draw=none,fill=none] at (0,-2.45) {(a) A schematic friendship graph};
		\end{tikzpicture}
		\hspace{0.2cm}
		\begin{tikzpicture}[scale=1.05, every node/.style={circle,fill=red!50!black,inner sep=1.7pt}]
			\node (L) at (0,0) {};
			\node (R) at (2,0) {};
			
			% Label for the right node
			%\node[draw=none, fill=none, xshift=15pt, yshift=5pt] at (R) {\Large $v$};
			
			% Left-side "pages"
			\node (TL1) at (-2,1.5) {};
			\node (TL2) at (0,1.5) {};
			\node (BL1) at (-2,-1.5) {};
			\node (BL2) at (0,-1.5) {};
			
			% Right-side "pages"
			\node (TR1) at (2,1.5) {};
			\node (TR2) at (4,1.5) {};
			\node (BR1) at (2,-1.5) {};
			\node (BR2) at (4,-1.5) {};
			
			% Drawing the edges
			\draw[thick] (L) -- (R); % The Spine
			
			% Page 1 (Top Left)
			\draw[thick] (TL1) -- (TL2) -- (R) -- (L) -- (TL1);
			
			% Page 2 (Bottom Left)
			\draw[thick] (BL1) -- (BL2) -- (R) -- (L) -- (BL1);
			
			% Page 3 (Top Right)
			\draw[thick] (TR2) -- (TR1) -- (L) -- (R) -- (TR2);
			
			% Page 4 (Bottom Right)
			\draw[thick] (BR2) -- (BR1) -- (L) -- (R) -- (BR2);
			\node[draw=none,fill=none] at (1.5,-2.45) {(b) A schematic book graph};
		\end{tikzpicture}
		\vspace{-70pt} % Adjust this value (e.g., -5pt, -15pt) to your liking
		\caption{A friendship $F_{3}$, and a book graph $B_{4}$.}
		\label{fig:basicfamilies}
	\end{figure}
	
	We now record the known formulas needed later.
	
	The following result gives the domination polynomial formula for union and join of graphs.
	\begin{proposition}[\cite{Alikhani2013Operations,KotekPreenSimonTittmannTrinks2012}]
		If $G_1$ and $G_2$ are graphs of orders $n_1$ and $n_2$, then
		$$
		D(G_1\cup G_2,x)=D(G_1,x)D(G_2,x),
		$$
		and
		$$
		D(G_1+G_2,x)=\bigl((1+x)^{n_1}-1\bigr)\bigl((1+x)^{n_2}-1\bigr)+D(G_1,x)+D(G_2,x).
		$$
	\end{proposition}
	
	Next result gives the known formula for the domination polynomial of friendship graphs.
	\begin{theorem}[\cite{AlikhaniBrownJahari2016}]
		\label{thm:knownFn}
		For every $n\ge 1$,
		$$
		D(F_n,x)=(2x+x^2)^n+x(1+x)^{2n}.
		$$
	\end{theorem}
	
	The following formula gives the domination polynomial for book graphs.
	\begin{theorem}[\cite{AlikhaniBrownJahari2016}]
		\label{thm:knownBn}
		For every $n\ge 1$,
		$$
		D(B_n,x)=(x^2+2x)^n(2x+1)+x^2(1+x)^{2n}-2x^n.
		$$
	\end{theorem}

	\begin{definition}[Limit of roots]
		Let $\{f_n(x)\}$ be a sequence of polynomials. A complex number $z$ is called a \emph{limit of roots} if either $f_n(z)=0$ for all sufficiently large $n$, or $z$ is a limit point of the union of all root sets $Z(f_n)$.
	\end{definition}
	
	The following theorem gives the alternative idea for finding the limit of the zeros of a polynomial.
	\begin{theorem}[Beraha--Kahane--Weiss, restated from \cite{BerahaKahaneWeiss1978,BrownHickman2002}]
		\label{thm:BKW}
		Suppose
		$$
		f_n(x)=\alpha_1(x)\lambda_1(x)^n+\cdots+\alpha_k(x)\lambda_k(x)^n,
		$$
		where the $\alpha_j$ and $\lambda_j$ are fixed nonzero polynomials and no pair $\lambda_i,\lambda_j$ differs by a unimodular constant factor identically. Then $z$ is a limit of roots of ${f_n}$ if and only if either:
		\begin{enumerate}[label=\textup{(\roman*)}]
			\item at least two of the numbers $|\lambda_j(z)|$ are equal and strictly larger than all the others, or
			\item for some $j$, the modulus $|\lambda_j(z)|$ is strictly larger than all others and $\alpha_j(z)=0$.
		\end{enumerate}
	\end{theorem}
	
	The following problems were asked in \cite{AlikhaniBrownJahari2016}.
	\begin{openproblem}\label{problem}
		The following problems about the domination polynomial are open.
		\begin{enumerate}[label=\textup{(\roman*)}]
			\item (Question 3.1 \cite{AlikhaniBrownJahari2016}) For $n$ even, does $F_{n}$ have exactly three real roots?
			\item (Question 3.2 \cite{AlikhaniBrownJahari2016}) What is a good upper bound on the modulus of the roots of $F_{n}$?
			\item (Question 3.3 \cite{AlikhaniBrownJahari2016}) What can be said about the domination roots of book graphs?
			\item (Question 3.6 \cite{AlikhaniBrownJahari2016}) Is $-2$ the only possible nonzero integer domination root?
		\end{enumerate}	
	\end{openproblem}
	
	 The paper \cite{AlikhaniBrownJahari2016} showed that when $n$ is odd the friendship graph $F_n$ has no nonzero real domination root, and when $n$ is even it has at least two nonzero real roots in $(-2,0)$. It also identified the limiting hyperbola for the complex roots of $F_n$ and derived the explicit formula for $D(B_n,x)$. However, the exact number of real roots for even $n$, a usable upper bound on the modulus of domination roots of $F_n$, and a systematic root-limit theorem for the book graphs were not supplied. The integer-root problem also remained open in general. These are precisely the gaps addressed  in the present paper.
	
	\section{Question 3.1: Exact number of real roots of $F_n$}\label{section 3}
	
	This section gives the exact real-root count for friendship graphs when $n$ is even.
	
	The following result gives the affirmative answer  Question~3.1 from \cite{AlikhaniBrownJahari2016}.
	\begin{theorem} \label{thm:exactthree}
		If $n$ is even, then $D(F_n,x)$ has exactly three real zeros, namely $0$ and two simple zeros
		$$
		x_n^- \in (-2,-1), \qquad x_n^+ \in (-1,0).
		$$
	\end{theorem}
	
	\begin{proof}
		By Theorem~\ref{thm:knownFn}, the domination polynomial of friendship graph is
		$$
		D(F_n,x)=x^n(x+2)^n+x(1+x)^{2n}.
		$$
		Given that $n$ is even, the established sign argument from \cite{AlikhaniBrownJahari2016} demonstrates that any nonzero real root is contained inside the interval $(-2,0)$. Therefore, it is sufficient to demonstrate that there exists precisely one root in each interval $(-2,-1)$ and $(-1,0)$.
		
		\medskip
		First we show uniqueness of $x_n^+$ in $(-1,0)$. Let $x=t-1$ with $0<t<1.$ Then the above polynomial can be put as
		$$
		D(F_n,t-1)=(t^2-1)^n+(t-1)t^{2n}.
		$$
		Given that $n$ is even, so $ (t^2-1)^n=(1-t^2)^n, $ and the equation $D(F_n,t-1)=0$ becomes $ 1-t^2)^n=(1-t)t^{2n}.$
		Now, using $1-t^2=(1-t)(1+t)$, and dividing by $1-t>0$, we obtain
		$(1-t)^{,n-1}(1+t)^n=t^{2n}.$ Next define
		$$
		\phi_n(t)=2n\log t-(n-1)\log(1-t)-n\log(1+t), \qquad 0<t<1.
		$$
		Then $D(F_n,t-1)=0$ if and only if $\phi_n(t)=0$. Also, differentiating $\phi_n(t),$ we have
		$$
		\phi_n'(t)=\frac{2n}{t}+\frac{n-1}{1-t}-\frac{n}{1+t}.
		$$ 
		Since $\tfrac{2n}{t}>0$, and $\tfrac{n-1}{1-t}-\tfrac{n}{1+t}>-\tfrac{1}{1+t}$, so we have $\phi_n'(t)>0$ on $(0,1)$. Thus $\phi_n$ is strictly increasing. Aslo, we have
		$$
		\lim_{t\to 0^+}\phi_n(t)=-\infty, \qquad \lim_{t\to 1^-}\phi_n(t)=+\infty.
		$$
		So, there is a unique $t_n^+\in(0,1)$ satisfying $\phi_n(t_n^+)=0$. Consequently there is a unique root
		$$
		x_n^+=t_n^+-1\in(-1,0).
		$$
		
		Let \( x = t - 1 \), and define \( A_n(t) =t^{2n} - (1 - t)^{n-1}(1 + t)^n. \) Consequently, \[ D(F_n, t - 1) = -(1 - t)A_n(t), \qquad \text{and}\qquad  A_n(t) = t^{2n}(1 - e^{-\phi_n(t)}). \] 	If \( t_n^+ \) represents the only root of \( \phi_n \), then \( \phi_n(t_n^+) = 0 \), and  \( A_n'(t_n^+) = t_n^{2n}\phi_n'(t_n^+) > 0. \) Consequently, \( A_n \) possesses a simple zero at \( t_n^+ \), leading to the conclusion that \( x_n^+ = t_n^+ - 1 \) 	constitutes a trivial zero of \( D(F_n, x) \).
		
		\medskip
		For uniqueness of $x_n^- $ in $(-2,-1)$, we let $x=-1-s $ with $ 0<s<1.$ So, the domination polynomial is
		$$
		D(F_n,-1-s)=(s^2-1)^n-(1+s)s^{2n}.
		$$
		Since $n$ is even, so $	(s^2-1)^n=(1-s^2)^n.$ So, $D(F_n,-1-s)=0$ is equivalent to $ (1-s^2)^n=(1+s)s^{2n}.$ 
		Now, with $1-s^2=(1-s)(1+s)$ and dividing by $1+s>0$, we obtain $(1-s)^n(1+s)^{n-1}=s^{2n}.$ Consider the function
		$$
		\psi_n(s)=2n\log s-n\log(1-s)-(n-1)\log(1+s), \qquad 0<s<1.
		$$
		Thus, $D(F_n,-1-s)=0$ if and only if $\psi_n(s)=0$. Thus, by differentiating, we obtain
		$$
		\psi_n'(s)=\frac{2n}{s}+\frac{n}{1-s}-\frac{n-1}{1+s}>0
		$$
		for all $s\in(0,1)$. Further $\lim_{s\to 0^+}\psi_n(s)=-\infty, $ and $ \lim_{s\to 1^-}\psi_n(s)=+\infty. $
		Hence, there is a unique $s_n\in(0,1)$ with $\psi_n(s_n)=0$, and therefore a unique root
		$
		x_n^-=-1-s_n\in(-2,-1).
		$
		As above, for \( x = -1 - s \), we define $ B_n(s) =s^{2n} - (1-s)^n(1+s)^{n-1}. $ Consequently, \[ D(F_n, -1 - s) = -(1+s)B_n(s), \qquad  \text{and} \qquad B_n(s) = s^{2n}(1 - e^{-\psi_n(s)}). \] If \( s_n \) is the only root of \( \psi_n \), then \[ B'_n(s_n) = s_n^{2n}\psi'_n(s_n) > 0, \] and therefore $ x_n^{-} = -1 - s_n $ is a simple zero.
		\medskip
		
		The combination of the two uniqueness assertions, along with the established exclusion of nonzero real roots beyond the interval $(-2,0)$, results in precisely two nonzero real roots. As $x=0$ is a root of $D(F_n,x)$, the polynomial possesses precisely three real roots.
	\end{proof}
	
	The following consequences shows that the real roots decrease to $-1\pm\frac{1}{\sqrt2}.$
	\begin{corollary}
		\label{cor:monotone}
		For even $n$, the roots from Theorem~\ref{thm:exactthree} satisfy
		$$
		x_n^- \searrow -1-\frac{1}{\sqrt2},
		\qquad
		x_n^+ \searrow -1+\frac{1}{\sqrt2}
		$$
		along the subsequence of even integers $n=2,4,6,\dots$.
	\end{corollary}
	
	\begin{proof}
		Let $t_n=x_n^++1\in(0,1)$ and $s_n=-(x_n^-+1)\in(0,1)$. Then from the proof of Theorem~\ref{thm:exactthree}, we have
		$$
		(1-t_n)^{n-1}(1+t_n)^n=t_n^{2n},\quad \text{and} \quad 	(1-s_n)^n(1+s_n)^{n-1}=s_n^{2n}.
		$$
		 Clearly, we see that $t_n>\tfrac{1}{\sqrt{2}}$. The initial identity shows that \(\left(\frac{t_n^2}{1-t_n^2}\right)^n=\frac{1}{1-t_n}>1 \), therefore \(\frac{t_n^2}{1-t_n^2}>1 \), implying that \(t_n>\frac{1}{\sqrt{2}} \).  From Theorem \ref{thm:exactthree}, the difference of functions $\phi_n$ and $\phi_{n+2}$ is
		 \begin{equation*}
		 	\phi_{n+2}(t)-\phi_n(t)=2\log\left(\frac{t^2}{1-t^2}\right).
		 \end{equation*}
		 At \(t=t_n \), the difference \(\phi_{n+2}(t)-\phi_n(t) \) is positive, indicating that \(t_n > \frac{1}{\sqrt{2}} \). Because $\phi_{n+2}$ is absolutely expanding and $\phi_{n+2}(t_n) > 0$, the zero must be to the left of $t_n$. Thus, $t_{n+2} < t_n$, implying $x_{n+2}^+ < x_n^+$.
		With a similar idea, the second identity implies that $s_n<\tfrac{1}{\sqrt2}$, since
		$$
		\left(\frac{s_n^2}{1-s_n^2}\right)^n=\frac{1}{1+s_n}<1.
		$$
		Theorem \ref{thm:exactthree} implies that
		$$
		\psi_{n+2}(s)-\psi_n(s)=2\log\left(\frac{s^2}{1-s^2}\right).
		$$
		At $s=s_n$, the value $\psi_{n+2}(s)-\psi_n(s)$ is negative,  so $\psi_{n+2}(s_n)<0$, and since $\psi_{n+2}$ is strictly rising, its zero is to the right of $s_n$. Thus, $s_{n+2}>s_n$, and it implies that $x_{n+2}^-<x_n^-$. The sequences ${t_n}$ and ${s_n}$ are monotone, bounded, and hence convergent. To evaluate the limit in $\left(\frac{t_n^2}{1-t_n^2}\right)^n=\frac{1}{1-t_n}$, the only solution is $\tfrac{t^2}{1-t^2}=1$, which gives $t=\tfrac{1}{\sqrt2}$.
		 Likewise, $s_n\to \tfrac{1}{\sqrt2}$. Thus, we get
		$$
		x_n^+=t_n-1 \to -1+\frac{1}{\sqrt2},
		\qquad\text{and} \qquad
		x_n^-=-1-s_n \to -1-\frac{1}{\sqrt2}.
		$$
		The demonstrated monotonicity indicates that the convergence is decreasing in both scenarios.
	\end{proof}
	
	 The preceding literature established the existence of at least two nonzero real roots for even \( n \) and identified the limiting hyperbola \cite{AlikhaniBrownJahari2016}, although it failed to furnish a proof for the exact enumeration of real roots. Theorem~\ref{thm:exactthree} entirely resolves that gap. Corollary~\ref{cor:monotone} further enhances the previous asymptotic framework by providing monotone convergence of the two real branches.

	The values shown in Table \ref{tab:friendshipreal}  persist in progressing monotonically towards the hyperbola intercepts
	$$-1 - \frac{1}{\sqrt{2}} \approx -1.707106781,\qquad
	-1+\frac{1}{\sqrt{2}} \approx -0.292893219.$$ The augmented data elucidate the convergence fashion and substantiate the monotonicity assertion demonstrated in Corollary~\ref{cor:monotone}.
	\begin{table}[H]
		\centering
		\caption{Even friendship graphs: exact real-root count and limiting values for $2\le n\le 20$.}
		\label{tab:friendshipreal}
		\begin{tabular}{cccc}
			\toprule
			$n$ & $x_n^-$ & $x_n^+$ & Limit points \\
			\midrule
			$2$  & $-1.660992532$ & $-0.151625104$ & \multirow{10}{*}{$-1\pm \tfrac{1}{\sqrt2}$} \\
			$4$  & $-1.683727169$ & $-0.231617585$ & \\
			$6$  & $-1.691458147$ & $-0.253745968$ & \\
			$8$  & $-1.695348455$ & $-0.264127671$ & \\
			$10$ & $-1.697690028$ & $-0.270155995$ & \\
			$12$ & $-1.699254020$ & $-0.274095001$ & \\
			$14$ & $-1.700372557$ & $-0.276870573$ & \\
			$16$ & $-1.701212207$ & $-0.278931884$ & \\
			$18$ & $-1.701865704$ & $-0.280523234$ & \\
			$20$ & $-1.702388772$ & $-0.281788907$ & \\
			\bottomrule
		\end{tabular}
	\end{table}
	Table~\ref{tab:friendshipreal} corroborates the precise count from Theorem~\ref{thm:exactthree} and elucidates the monotonic progression of the two nonzero real roots towards their limiting values.
	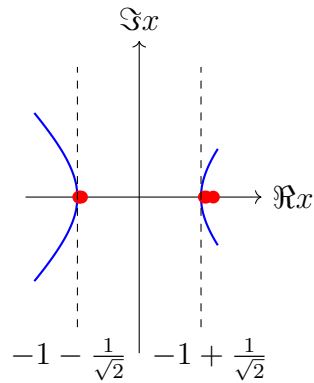
\begin{figure}[H]
		\centering
		\begin{tikzpicture}[scale=1.15]
			% Axes
			\draw[->] (-2.3,0) -- (0.4,0) node[right] {$\Re x$};
			\draw[->] (-1,-1.8) -- (-1,1.8) node[above] {$\Im x$};
			
			% Left Branch: x must be <= -1.7071
			\draw[blue,thick,domain=-2.2:-1.708,smooth,samples=50,variable=\x]
			plot ({\x},{sqrt((\x+1)*(\x+1)-0.5)});
			\draw[blue,thick,domain=-2.2:-1.708,smooth,samples=50,variable=\x]
			plot ({\x},{-sqrt((\x+1)*(\x+1)-0.5)});
			
			% Right Branch: x must be >= -0.2929
			\draw[blue,thick,domain=-0.292:-0.1,smooth,samples=50,variable=\x]
			plot ({\x},{sqrt((\x+1)*(\x+1)-0.5)});
			\draw[blue,thick,domain=-0.292:-0.1,smooth,samples=50,variable=\x]
			plot ({\x},{-sqrt((\x+1)*(\x+1)-0.5)});
			
			% Points and Asymptotes
			\fill[red] (-1.661,0) circle (2.1pt);
			\fill[red] (-1.684,0) circle (2.1pt);
			\fill[red] (-1.691,0) circle (2.1pt);
			\fill[red] (-0.152,0) circle (2.1pt);
			\fill[red] (-0.232,0) circle (2.1pt);
			\fill[red] (-0.254,0) circle (2.1pt);
			
			\draw[dashed] (-1.7071,-1.5) -- (-1.7071,1.5);
			\draw[dashed] (-0.2929,-1.5) -- (-0.2929,1.5);
			
			\node[below] at (-1.8071,-1.5) {$-1-\tfrac{1}{\sqrt2}$};
			\node[below] at (-0.1929,-1.5) {$-1+\tfrac{1}{\sqrt2}$};
		\end{tikzpicture}
		\caption{Schematic real-axis location of the nonzero roots of $D(F_n,x)$ for even $n$, together with the limiting hyperbola intercepts.}
		\label{fig:friendshiphyperbola}
	\end{figure}
	
	  Figure~\ref{fig:friendshiphyperbola} illustrates how the two real branches move monotonically toward the real-axis intersections of the limiting hyperbola. The denser arrangement of plotted dots enhances the visibility of the monotone convergence towards $-1\pm \tfrac{1}{\sqrt2}$. The location of zeros of $D(F_n,z)$ for $n=10$ is shown in Figure \ref{fig1} (a).
	  
	  \begin{figure}[H]
	  	\centering
	  	% Subfigure A
	  	\begin{subfigure}[b]{0.45\textwidth}
	  		\centering
	  		\includegraphics[width=\textwidth]{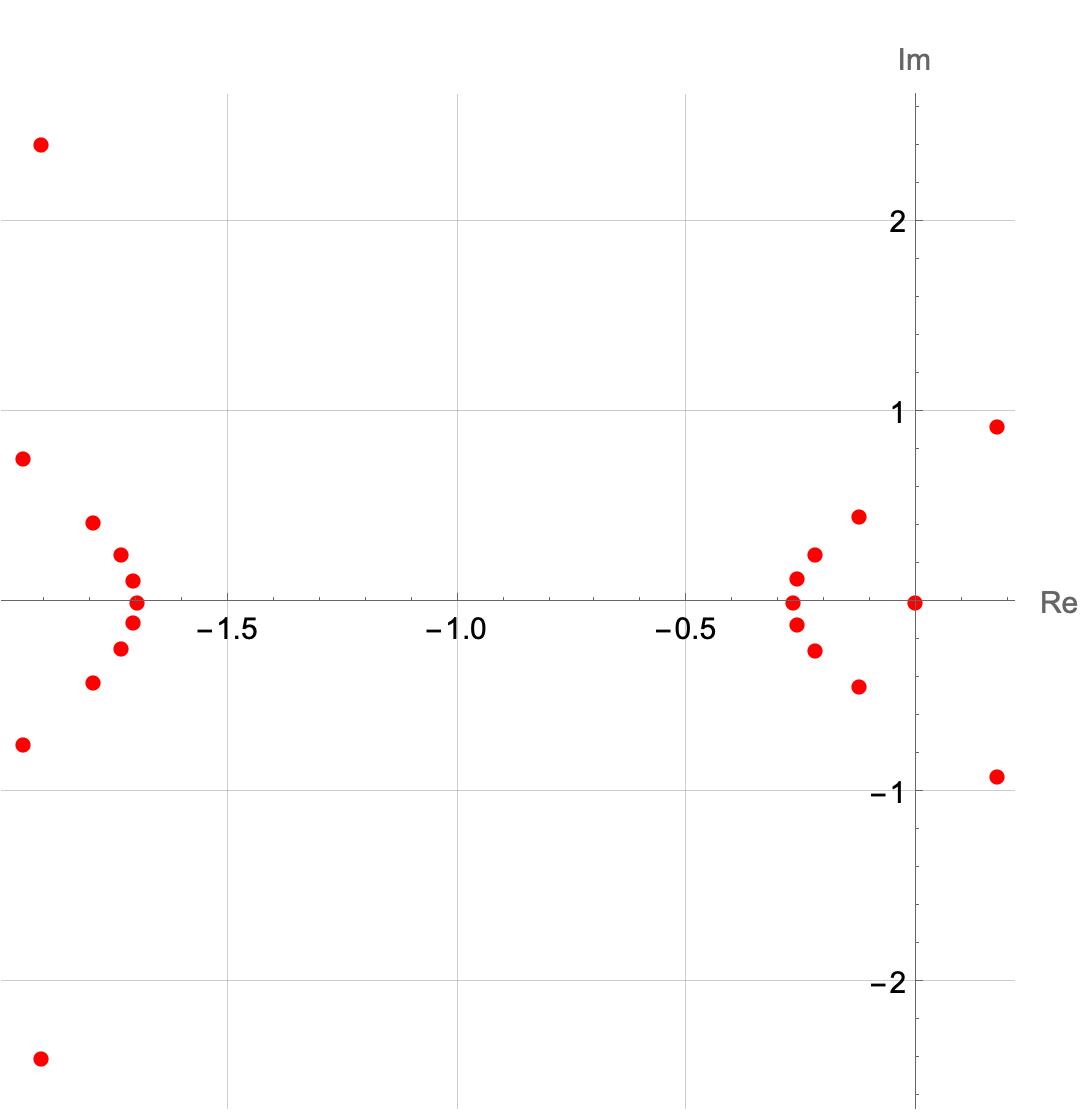}
	  		\caption{Zeros of $D(F_{10},z)$}
	  		\label{fig:p1}
	  	\end{subfigure}
	  	\hfill % This pushes the two images to the edges
	  	% Subfigure B
	  	\begin{subfigure}[b]{0.45\textwidth}
	  		\centering
	  		\includegraphics[width=\textwidth]{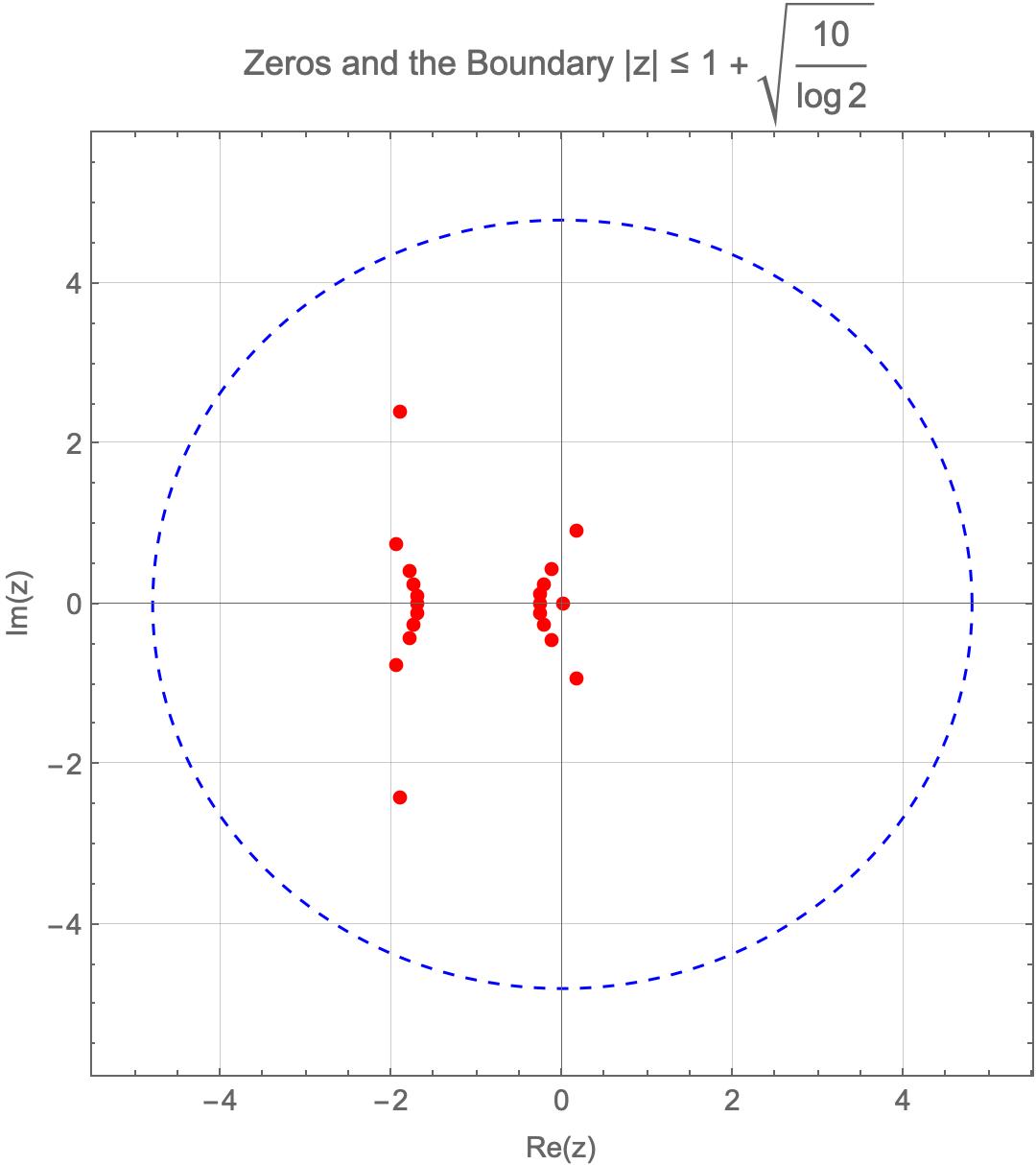}
	  		\caption{Annular region $|z|\leq 1+\sqrt{\frac{10}{\log 2}}.$}
	  		\label{fig:p2}
	  	\end{subfigure}
	  	
	  	\caption{Detailed analysis of the zeros for $D(F_{10},z)$.}
	  	\label{fig1}
	  \end{figure}
	
	\section{Question 3.2: Upper bounds on the modulus of the roots of $F_n$}\label{section 4}
	
	This section gives a quantitative answer to the root-modulus problem for friendship graphs.
	
	\begin{theorem}
		\label{thm:implicitbound}
		Let $z\in \mathbb{C}$ be a nonzero domination root of $F_n$. If $|z|>1$, then
		$$
		(|z|-1)^2\log |z| \le n.
		$$
	\end{theorem}
	
	\begin{proof}
		If $z\neq 0$ satisfy $D(F_n,z)=0$, then by Theorem~\ref{thm:knownFn}, the domination polynomial of $F_n$ is
		$$
		z^n(z+2)^n+z(1+z)^{2n}=0.
		$$
		Since $z\neq 0$, so above equation cab be written as
		$$
		z\left(\frac{z(z+2)}{(1+z)^2}\right)^n=-1.
		$$
		Now, with  identity $\frac{z(z+2)}{(1+z)^2}=1-\frac{1}{(1+z)^2},$ we obtain
		$
		z=-\left(1-\frac{1}{(1+z)^2}\right)^n,
		$
		and thereby we have
		$$
		|z|^{1/n}=\left|1-\frac{1}{(1+z)^2}\right|.
		$$
		Now, let's say that |z| > 1. Then, by the triangle inequality and the reverse triangle inequality, we obtain
		$$
		|z|^{1/n}\le 1+\frac{1}{|1+z|^2}\le 1+\frac{1}{(|z|-1)^2}.
		$$
		By using logarithms, we obtain
		$$
		\frac{1}{n}\log |z|
		\le
		\log\left(1+\frac{1}{(|z|-1)^2}\right)
		\le
		\frac{1}{(|z|-1)^2},
		$$
		since $\log(1+u)\le u$ for $u>-1$. Multiplying with $n(|z|-1)^2$ yields the result.
	\end{proof}

	The following consequence is immediate from the above result.
	\begin{corollary}
		\label{cor:explicitbound}
		Every domination root $z$ of $F_n$ satisfies
		$$
		|z|\le 1+\sqrt{\frac{n}{\log 2}}.
		$$
	\end{corollary}
	
	\begin{proof}
		If $|z|\le 2$, then the conclusion is immediate as
		$$
		2 \le 1+\sqrt{\frac{1}{\log 2}} \le 1+\sqrt{\frac{n}{\log 2}}.
		$$
		Now for $|z|>2$, Theorem~\ref{thm:implicitbound} is applicable, and with $\log |z|\ge \log 2$, we have
		$$
		(|z|-1)^2\log 2 \le (|z|-1)^2\log |z| \le n.
		$$
		Hence, we obtain $|z|-1 \le \sqrt{\frac{n}{\log 2}}.$
	\end{proof}
	
	\begin{remark}
		The implicit inequality from Theorem~\ref{thm:implicitbound} is more precise than the explicit estimate in Corollary~\ref{cor:explicitbound}. As, if $R_n>1$ is the unique positive solution of $ (R_n-1)^2\log R_n=n, $ then every domination root of $F_n$ satisfies $|z|\le R_n.$ Figure \ref{fig1} (b) shows the location of zeros of $D(F_{10},z)$ with annular region $|z|1+\sqrt{\frac{10}{\log 2}}.$
	\end{remark}
	
	 Question 3.2 in \cite{AlikhaniBrownJahari2016} asked about an effective upper bound for the modulus of the roots of \( F_n \).  Theorem~\ref{thm:implicitbound} and Corollary~\ref{cor:explicitbound} present bounds in a clear analytical format. The growth is at most of order $\sqrt n$, significantly above the linear bounds derived from rudimentary coefficient estimations.
	
	 	For specific values of $n$, the observed maximum modulus of a root of $D(F_n,x)$ is significantly less than the bound established in Corollary~\ref{cor:explicitbound}.		
		 \begin{table}[H]
		 	\centering
		 	\caption{Comparison of the observed largest root modulus of $D(F_n,x)$ with the bound $1+\sqrt{\tfrac{n}{\log 2}}$.}
		 	\label{tab:modulusbound}
				\begin{tabular}{ccc}
				\toprule
				$n$ & $\max{|z|: D(F_n,z)=0}$ & $1+\sqrt{\tfrac{n}{\log 2}}$ \\
				\midrule
				$2$  & $1.992648074$ & $2.698003$ \\
				$4$  & $2.356603725$ & $3.402245$ \\
				$6$  & $2.632721118$ & $3.942700$ \\
				$8$  & $2.863372765$ & $4.396007$ \\
				$10$ & $3.065002825$ & $4.795831$ \\
				$12$ & $3.246059018$ & $5.157249$ \\
				$14$ & $3.411562628$ & $5.489694$ \\
				$16$ & $3.564784750$ & $5.799024$ \\
				$18$ & $3.707996013$ & $6.089865$ \\
				$20$ & $3.842848052$ & $6.365913$ \\
				\bottomrule
			\end{tabular}
			\end{table}

		Thus the new estimate is rigorous and reasonably close to the numerically observed scale, although it is not expected to be asymptotically sharp. 
	
	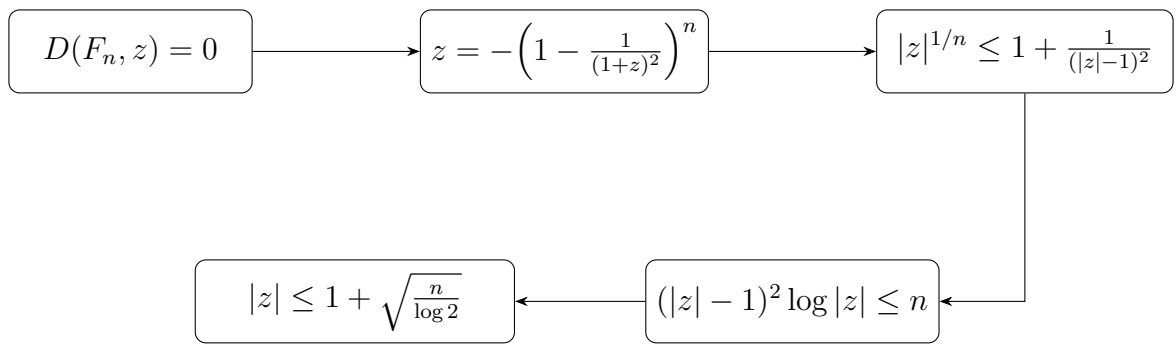
\begin{figure}[H]
		\centering
		\begin{tikzpicture}[node distance=2.2cm,>=Stealth]
			\node[draw,rounded corners,align=center,minimum width=3.2cm,minimum height=1.1cm] (a) {$D(F_n,z)=0$};
			\node[draw,rounded corners,align=center,minimum width=3.7cm,minimum height=1.1cm,right=of a] (b) {$z=-\Bigl(1-\frac{1}{(1+z)^2}\Bigr)^n$};
			\node[draw,rounded corners,align=center,minimum width=3.9cm,minimum height=1.1cm,right=of b] (c) {$|z|^{1/n}\le 1+\frac{1}{(|z|-1)^2}$};
			\node[draw,rounded corners,align=center,minimum width=3.6cm,minimum height=1.1cm,below=of b, xshift=3cm] (d) {$(|z|-1)^2\log |z|\le n$};
			\node[draw,rounded corners,align=center,minimum width=4.2cm,minimum height=1.1cm,right=of d, xshift=-12cm] (e) {$|z|\le 1+\sqrt{\frac{n}{\log 2}}$};
			\draw[->] (a)--(b);
			\draw[->] (b)--(c);
			\draw[->] (c)|-(d);
			\draw[->] (d)--(e);
		\end{tikzpicture}
		\caption{Block diagram of the proof of the modulus bound for friendship graphs.}
		\label{fig:blockbound}
	\end{figure}
	
	 Table \ref{tab:modulusbound} demonstrates that the proved estimate is quantitatively meaningful. Figure~\ref{fig:blockbound} records the logical chain of the argument: the key identity for a root is converted into a modulus inequality, then into a logarithmic estimate, and finally into an explicit bound.
	
	\section{Question 3.3: What can be said about the domination roots of book graphs?}\label{section 5}
	
	This section provides a precise asymptotic characterization of the domination roots of $B_n$ and some specific implications about real roots.
	
	\begin{theorem}
		\label{thm:booklimit}
		Let $\mathcal{L}_B$ denote the set of limits of domination roots of the book graphs $B_n$. Then
		$
		\mathcal{L}_B=\{0,-\tfrac12\}\cup C_{12}\cup C_{13}\cup C_{23},
		$
		where
		$
		C_{12}=\bigl\{x\in\mathbb{C}: |x(x+2)|=|x+1|^2 \ge |x|\bigr\}, C_{13}=\bigl\{x\in\mathbb{C}: |x+2|=1,\ |x|\ge |x+1|^2\bigr\},
		$
		and
		$C_{23}=\bigl\{x\in\mathbb{C}: |x+1|^2=|x|,\ |x+2|\le 1\bigr\}.$
		Equivalently, in Cartesian coordinates $x=a+ib\in \mathbb{C}$, the set $C_{12}$ is the part of the hyperbola
		$
		(a+1)^2-b^2=\frac12
		$
		for which $|x+1|^2\ge |x|$, the set $C_{13}$ is an arc of the circle
		$
		(a+2)^2+b^2=1,
		$
		and the set $C_{23}$ is an arc of the quartic curve
		$
		\bigl((a+1)^2+b^2\bigr)^2=a^2+b^2
		$
		lying inside the disk $|x+2|\le 1$.
	\end{theorem}
	
	\begin{proof}
		By Theorem~\ref{thm:knownBn}, the domination polynomial of the book graph is
		$$
		D(B_n,x)=(2x+1)\bigl(x(x+2)\bigr)^n+x^2\bigl((x+1)^2\bigr)^n-2x^n.
		$$
		Writing above polynomial in the form of Theorem  \ref{thm:BKW}, we have
		$$
		D(B_n,x)=\alpha_1(x)\lambda_1(x)^n+\alpha_2(x)\lambda_2(x)^n+\alpha_3(x)\lambda_3(x)^n
		$$
		with $\alpha_1(x)=2x+1, \lambda_1(x)=x(x+2),\alpha_2(x)=x^2,\lambda_2(x)=(x+1)^2, \alpha_3(x)=-2, $ and $\lambda_3(x)=x.$
		The three $\lambda_j$ are nonzero polynomials, and no pair is identically differs by a unimodular constant factor. Consequently, Theorem~\ref{thm:BKW} is applicable, and we consider the following cases.
		
		\medskip
		
		\noindent(1) Exact roots of $D(B_n,x)$ for all $n$.  Since every term in \(D(B_n,x)\) is divisible by \(x^2\) for \(n \geq 2\), and noting that \(B_1=C_4\) also has \(0\) as a double domination root, it implies that \(x=0\) is a root for all \(n\). Consequently, $0 \in \mathcal{L}_B$.
		
		\medskip
		
		\noindent(2) The coefficient of $D(B_n,x)$ vanishing with unique dominant term. The only feasible zeros of the coefficients are $x=-\tfrac{1}{2}$ for $\alpha_1$ and $x=0$ for $\alpha_2$. The point $x=0$ has been addressed. At $x=-\tfrac12$, we have
		$$
		|\lambda_1(-\tfrac12)|=\frac34,\qquad
		|\lambda_2(-\tfrac12)|=\frac14,\qquad
		|\lambda_3(-\tfrac12)|=\frac12.
		$$
		Consequently, $\lambda_1$ possesses the largest modulus at there, and since $\alpha_1(-\tfrac{1}{2})=0$, so Theorem~\ref{thm:BKW}\textup{(ii)} implies that $-\tfrac{1}{2}\in\mathcal{L}_B$.
		
		\medskip
		
		\noindent(3) The equality of two dominant moduli of $D(B_n,x)$. We consider the three possible pairings.  (a) Pair $(\lambda_1,\lambda_2)$. The equality $|\lambda_1(x)|=|\lambda_2(x)|$ implies that $|x(x+2)|=|x+1|^2.$ Since $x(x+2)=(x+1)^2-1$, so it is equivalent to $|(x+1)^2-1|=|(x+1)^2|.$ Let $x=a+ib$ and $x+1=u+ib$ with $u=a+1$, and squaring both sides, we obtain
		$u^2-b^2=\frac12,$ 	that is, $ (a+1)^2-b^2=\frac12.$ To satisfy Theorem~\ref{thm:BKW}\textup{(i)}, the common modulus must exceed the third modulus. Thus, we also necessitate $ |x+1|^2\ge |x|. $ This gives precisely the set $C_{12}$.\\
		(b) Pair $(\lambda_1,\lambda_3)$. The equality $|\lambda_1(x)|=|\lambda_3(x)|$ gives $|x(x+2)|=|x|. $ Hence either $x=0$ or $|x+2|=1$. Since, $x=0$ has already been included separately, we consider $|x+2|=1.$ The common modulus is $|x|$, and dominance over $\lambda_2$ requires $|x|\ge |x+1|^2.$  This is the set $C_{13}$.\\
		(c) Pair $(\lambda_2,\lambda_3)$. The equality $|\lambda_2(x)|=|\lambda_3(x)|$ is equivalent to $|x+1|^2=|x|.$ The common modulus is $|x|$. Dominance over $\lambda_1$ requires $|x|\ge |x(x+2)|=|x||x+2|.$ For $x\neq 0$, this is equivalent to $|x+2|\le 1$. Hence, we obtain exactly $C_{23}={x: |x+1|^2=|x|,\ |x+2|\le 1}.$ All limit points described by Theorem~\ref{thm:BKW} are therefore contained in
		$$
		\{0,-\tfrac12\}\cup C_{12}\cup C_{13}\cup C_{23}.
		$$
	\end{proof}
	
	The following consequence is immediate from the above result.
	\begin{corollary}\label{cor:realintersections}
		The real-axis intersections of the limit set $\mathcal{L}_B$ are
		$$
		0,\quad -\frac12,\quad -1\pm \frac{1}{\sqrt2},\quad -1,\quad -3,\quad \frac{-3-\sqrt5}{2}.
		$$
	\end{corollary}
	
	\begin{proof}
		We consider each case separately. For $C_{12}$, set $b=0$ in $(a+1)^2-b^2=\tfrac12$ to obtain$a=-1\pm \frac{1}{\sqrt2}.$ For $C_{13}$, the real-axis intersections of $|x+2|=1$ are 	$x=-3,-1.$ 	For $C_{23}$, on the real axis the equation $|x+1|^2=|x|$ becomes $(x+1)^2=-x$ 	for negative $x$, and hence $x^2+3x+1=0.$ The two solutions are
		$x=\frac{-3\pm \sqrt5}{2}.$ Among these, the dominance condition $|x+2|\le 1$ keeps only $x=\frac{-3-\sqrt5}{2}.$ Together with the special points $0$ and $-\tfrac12$, this gives the required list.
	\end{proof}
	
	The following proposition gives the information about the real domination root of $D(B_n,x).$
	\begin{proposition}
		\label{prop:bookrealexistence}
		For even $n\ge 2$, the book graph $B_n$ has at least one real domination root in $(-\infty,-2)$ and at least one real domination root in $(-\tfrac12,0)$.
	\end{proposition}
	
	\begin{proof}
		By Theorem~\ref{thm:knownBn}, the domination polynomial of $B_{n}$ is 
		$$
		D(B_n,x)=(x^2+2x)^n(2x+1)+x^2(1+x)^{2n}-2x^n.
		$$
		As the degree is $2n+2$ and the leading coefficient is $1$, we have 
			$$
		\lim_{x\to -\infty} D(B_n,x)=+\infty.
		$$
		At $x=-2$, we have $D(B_n,-2)=4-2^{n+1}<0$	for every even $n\ge 2$. Hence, intermediate value theorem implies that a root lies in $(-\infty,-2)$. For even $n$, we have
		$$
		D\left(B_n,-\frac12\right)=\frac{1}{4^{n+1}}-\frac{2}{2^n}<0.
		$$
		 Conversely, $x=0$ is a double root, therefore the coefficient of $x^2$ in $D(B_n,x)$ is positive, hence $D(B_n,x)>0$ for all sufficiently small negative values of $x$. Therefore, there is a root in $(-\tfrac12,0)$.
	\end{proof}
	
	\begin{conjecture}
		\label{conj:bookreal}
		The numerical data suggest the following sharper parity-dependent statement.
		\begin{enumerate}[label=\textup{(\roman*)}]
			\item If $n$ is even, then $D(B_n,x)$ has exactly four real roots counting multiplicity: $0$ with multiplicity $2$, one simple root in $(-\infty,-2)$, and one simple root in $(-\tfrac12,0)$.
			\item If $n$ is odd, then $D(B_n,x)$ has exactly four real roots counting multiplicity: $0$ with multiplicity $2$ and two simple roots in $(-\tfrac12,0)$.
		\end{enumerate}
	\end{conjecture}
	
	 Question 3.3 from \cite{AlikhaniBrownJahari2016} inquired about the characteristics of the domination roots of book graphs. Theorem~\ref{thm:booklimit} provides a precise asymptotic solution by identifying the entire root-limit set through Beraha--Kahane--Weiss analysis \cite{BerahaKahaneWeiss1978}. Proposition~\ref{prop:bookrealexistence} gives specific information related to real roots. The precise complete categorization of real roots remains unresolved and is clearly stated as Conjecture~\ref{conj:bookreal}.
	
		The following table illustrates above results numerically for small values of $n$ in $D(B_n,x).$
	\begin{table}[H]
		\centering
		\caption{Numerical comparison for book graphs.}
		\label{tab:bookroots}
		\begin{tabular}{cccc}
			\toprule
			$n$ & Real roots of $D(B_n,x)$ & Predicted attracting set & Parity pattern \\
			\midrule
			$3$ & $-0.2451 ,-0.1563,0 ,0$ & near $-\tfrac12$ and hyperbola branch & odd \\
			$4$ & $-2.3560 ,-0.3708 ,0 ,0$ & near $C_{23}$ and $-\tfrac12$ & even \\
			$5$ & $-0.4001 ,-0.2188 ,0 ,0$ & near $-\tfrac12$ and hyperbola branch & odd \\
			$6$ & $-2.4247 ,-0.4334 ,0 ,0$ & near $C_{23}$ and $-\tfrac12$ & even \\
			$7$ & $-0.4526 ,-0.2422 ,0 ,0$ & near $-\tfrac12$ and hyperbola branch & odd \\
			$8$ & $-2.4651 ,-0.4672 ,0 ,0$ & near $C_{23}$ and $-\tfrac12$ & even \\
			$9$ & $-0.4773 ,-0.2549 ,0 ,0$ & near $-\tfrac12$ and hyperbola branch & odd \\
			$10$ & $-2.4917 ,-0.4844 ,0 ,0$ & near $C_{23}$ and $-\tfrac12$ & even \\
			\bottomrule
		\end{tabular}
	\end{table}
	Table~\ref{tab:bookroots} shows a clear parity-dependent pattern in the real domination roots of $B_n$: for odd $n$, the two nonzero real roots lie in $(-\tfrac12,0)$, while for even $n$, one root lies to the left of $-2$ and the other remains in $(-\tfrac12,0)$.
	
	The numerical values also agree with the theoretical limit-set description, indicating that one branch is attracted toward the region near $-\tfrac12$ and the hyperbolic component, whereas the even-indexed large negative root is influenced by the $C_{23}$ part of the limiting set. Table~\ref{tab:bookroots} shows that low-order numerical roots already lie near the predicted attracting sets, and that the parity effect in Proposition~\ref{prop:bookrealexistence} is visible from the first few examples.

	\begin{figure}[H]
		\centering
		\begin{tikzpicture}[scale=1.05]
			\draw[->] (-4.2,0) -- (1.2,0) node[right] {$a$};
			\draw[->] (0,-2.4) -- (0,2.4) node[above] {$b$};
			
			% circle |x+2|=1
			\draw[blue,thick] (-2,0) circle (1);
			
			% hyperbola branch (schematic) - Adjusted Domains
			% Left Branch: Must be <= -1.7071
			\draw[red,thick,domain=-3.1:-1.708,smooth,samples=50,variable=\x]
			plot ({\x},{sqrt((\x+1)*(\x+1)-0.5)});
			\draw[red,thick,domain=-3.1:-1.708,smooth,samples=50,variable=\x]
			plot ({\x},{-sqrt((\x+1)*(\x+1)-0.5)});
			
			% Right Branch: Must be >= -0.2929
			\draw[red,thick,domain=-0.292:0.7,smooth,samples=50,variable=\x]
			plot ({\x},{sqrt((\x+1)*(\x+1)-0.5)});
			\draw[red,thick,domain=-0.292:0.7,smooth,samples=50,variable=\x]
			plot ({\x},{-sqrt((\x+1)*(\x+1)-0.5)});
			
			% schematic quartic arc
			\draw[green!60!black,thick]
			(-2.62,0) .. controls (-2.45,0.55) and (-2.10,0.92) .. (-1.45,0.98);
			\draw[green!60!black,thick]
			(-2.62,0) .. controls (-2.45,-0.55) and (-2.10,-0.92) .. (-1.45,-0.98);
			
			% special points
			\fill[black] (0,0) circle (2pt) node[below right] {$0$};
			\fill[black] (-0.5,0) circle (2pt) node[below] {$-\tfrac12$};
			
			\node[blue] at (-2.0,1.35) {$C_{13}$};
			\node[red] at (0.55,1.75) {$C_{12}$};
			\node[green!60!black] at (-2.55,1.25) {$C_{23}$};
		\end{tikzpicture}
		\caption{Schematic illustration of the geometric components appearing in Theorem~\ref{thm:booklimit}}
		\label{fig:booklimit}
	\end{figure}
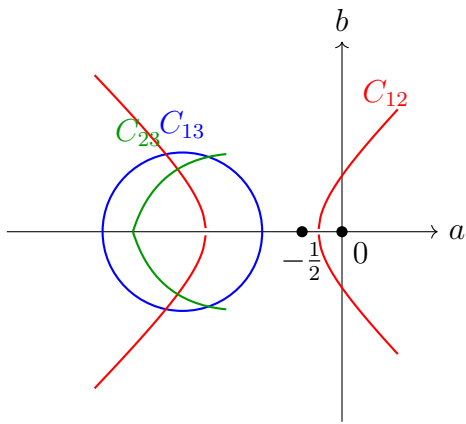
	
	 Figure~\ref{fig:booklimit} visualizes the three pieces appearing in Theorem~\ref{thm:booklimit}.  The circle $|x+2|=1$, the hyperbola $(a+1)^2-b^2=\tfrac12$, the quartic curve $|x+1|^2=|x|$, and the special points $0$ and $-\tfrac12$. The actual limit set consists only of the portions of these curves satisfying the dominance conditions defining $C_{12}$, $C_{13}$, and $C_{23}$.
	 It presents a geometric representation of the book-graph limit set, demonstrating the formation of its boundary by the intersection of circular, hyperbolic, and quartic algebraic curves within the complex plane.
	
	\section{Question 3.6: Integer domination roots}\label{section 6}
	
	This section gives a partial affirmative answer to the integer-root problem for the graph families studied in \cite{AlikhaniBrownJahari2016}.
	
	\begin{theorem}
		\label{thm:intfriend}
		For every $n\ge 1$, the friendship graph $F_n$ has no nonzero integer domination root.
	\end{theorem}
	
	\begin{proof}
		If $n$ is odd, then \cite{AlikhaniBrownJahari2016} showed that $F_n$ has no nonzero real domination root. Hence it has no nonzero integer domination root. Assume now that $n$ is even. By Theorem~\ref{thm:exactthree}, the only real roots of $D(F_n,x)$ are $0, x_n^- \in (-2,-1),$ and $x_n^+\in(-1,0).$ Neither interval contains an integer. Therefore, $0$ is the only integer domination root of $F_n$.
	\end{proof}
	
	\begin{theorem}
		\label{thm:intbook}
		For every $n\ge 1$, the book graph $B_n$ has no nonzero integer domination root.
	\end{theorem}
	
	\begin{proof}
		For every positive integer $m$, the definition of domination polynomial shows that
		$ D(B_n,m)>0, $ since all coefficients count dominating sets and are therefore nonnegative, with at least one positive coefficient. Hence no positive integer can be a root. We now validate negative integers. First, observe that
		$$
		D(B_n,-1)=(-1)(-1+2)^n+1\cdot 0^{2n}-2(-1)^n=
		\begin{cases}
			-3, & n \text{ even},\\
			1, & n \text{ odd},
		\end{cases}
		$$
		so it implies that $-1$ is never its root. Furthermore, $D(B_n,-2)=4-2(-2)^n,$ which is never zero for all $n\ge 1$.
		 Let $m\ge 3$. So from Theorem~\ref{thm:knownBn}, we derive
		$$
		D(B_n,-m)=\bigl(m(m-2)\bigr)^n(1-2m)+m^2(m-1)^{2n}-2(-m)^n.
		$$
		 For odd $n$, we have
		$$
		D(B_n,-m)=-(2m-1)m^n(m-2)^n+m^2(m-1)^{2n}+2m^n.
		$$
		As $(m-1)^2=m(m-2)+1>m(m-2),$  so we obtain
		$$
		m^2(m-1)^{2n}>m^2\bigl(m(m-2)\bigr)^n=m^{n+2}(m-2)^n.
		$$
		Thus, it follows that
		$$
		D(B_n,-m) >  m^n(m-2)^n\bigl(m^2-(2m-1)\bigr)+2m^n
		 m^n(m-2)^n(m-1)^2+2m^n>0.
		$$
		 For even $n$, we have
		$$
		D(B_n,-m)=-(2m-1)m^n(m-2)^n+m^2(m-1)^{2n}-2m^n.
		$$
		With inequality ideas as above, we have
		$$
		D(B_n,-m)		
		> m^n(m-2)^n\bigl(m^2-(2m-1)\bigr)-2m^n
		 m^n\Bigl((m-1)^2(m-2)^n-2\Bigr).
		$$
		Given that $m \ge 3$, it follows that $(m-1)^2(m-2)^n \ge 4$, so the final expression is positive. Therefore, $D(B_n,-m)>0$ for all $m\ge 3$.
		 Consequently, no negative number, except potentially $0$, can serve as a domination root. We ascertain that $0$ is the only integer domination root of $B_n$.
	\end{proof}
	
	Next, we have the following consequence related to the integer domination roof for friendship and the book graph.
	\begin{corollary}
		\label{cor:partialinteger}
		Within the graph families $ \{F_n:n\ge 1\}\cup \{B_n:n\ge 1\} $ the only integer domination root is $0$. Moreover, by Theorem~3.4 of \cite{AlikhaniBrownJahari2016}, the corona-type families $ B_{2m-1}\circ F_{2m-1}, B_{2m}\circ F_{2m+1},$ and $ B_{2m+1}\circ F_{2m}$ have $-2$ as their only nonzero real domination root. Consequently, among the natural families presently understood, the only nonzero integer domination root that actually occurs is $-2$.
	\end{corollary}
	
	\begin{proof}
		The first statement follows from Theorems~\ref{thm:intfriend} and \ref{thm:intbook}. The second statement is exactly the content of the cited corona-family theorem from \cite{AlikhaniBrownJahari2016}. Since every integer root is real, the final conclusion follows.
	\end{proof}
	
	\begin{openproblem}
		The global problem remains open: is $-2$ the only possible nonzero integer domination root of an arbitrary graph?
	\end{openproblem}
	
	 Question 3.6 in \cite{AlikhaniBrownJahari2016} was articulated for arbitrary graphs. Theorems~\ref{thm:intfriend} and \ref{thm:intbook} demonstrate that two essential dense families, friendship graphs and book graphs, provide no nonzero integer roots whatsoever. Corollary~\ref{cor:partialinteger} subsequently integrates these findings with established corona constructs to identify $-2$ as the sole nonzero integer root manifested in the most comprehensively understood families.
	
	 For illustration, we have
		$$
		D(F_4,-1) = 1,\qquad D(F_4,-2)=-2,\qquad D(F_4,-3)=1296-192=1104,
		$$
		so no nonzero integer root appears for $F_4$. Likewise, for book graph,
		$$
		D(B_4,-1)=-3,\qquad D(B_4,-2)=-28,\qquad D(B_4,-3)=1737,
		$$
		and again no nonzero integer root appears. On the other hand, the corona families from \cite{AlikhaniBrownJahari2016} do realize the integer root $-2$. 
	
	\begin{figure}[H]
		\centering
		\begin{tikzpicture}[scale=1.1]
			\draw[->] (-4.4,0) -- (1.2,0) node[right] {integer line};
			\foreach \x in {-4,-3,-2,-1,0,1}
			\draw (\x,0.08)--(\x,-0.08) node[below=4pt] {$\x$};
			
			\fill[blue] (0,0) circle (2.3pt);
			\node[blue,above] at (0,0) {$F_n,\ B_n$};
			
			\fill[red] (-2,0) circle (2.3pt);
			\node[red,above] at (-2,0) {corona families};
			
			\draw[decorate,decoration={brace,mirror,amplitude=6pt}] (-4,-0.65) -- (-1,-0.65);
			\node[below] at (-2.5,-0.95) {proved non-roots for $F_n$ and $B_n$};
			
			\draw[decorate,decoration={brace,amplitude=6pt}] (-2,0.65) -- (0,0.65);
			\node[above] at (-1,0.95) {realized integer roots in known families};
		\end{tikzpicture}
		\caption{Integer-root landscape for the families studied here: $0$ occurs for $F_n$ and $B_n$, while $-2$ appears in the previously known corona families.}
		\label{fig:integeraxis}
	\end{figure}
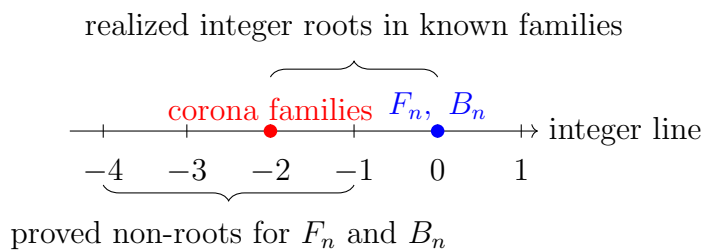
	 The preceding calculations provide direct assessments at small integers, whereas Figure~\ref{fig:integeraxis} encapsulates the qualitative overview derived in this section: friendship and book graphs yield solely the root $0$, whereas the corona constructions generate the nonzero integer root $-2$.
	
	\section{Conclusion and future work}\label{section 7}
	
	We have reexamined four unresolved inquiries on the domination polynomials of friendship and book graphs, yielding a combination of comprehensive and partial responses.
	
	For friendship graphs, the answer to  Open Problem \ref{problem}(i)  is complete (Theorem \ref{thm:exactthree}): if $n$ is even, then $D(F_n,x)$ has exactly three real roots, namely $0$ and two simple roots in $(-2,-1)$ and $(-1,0)$. We also proved that these roots move monotonically and converge to the hyperbola intercepts $-1\pm \tfrac{1}{\sqrt2}$ (Corollary \ref{cor:monotone}). For  Open Problem~\ref{problem}(ii) we derived the new inequality $(|z|-1)^2\log |z|\le n$ (Theorem \ref{thm:implicitbound}) for every nonzero domination root $z$ of $F_n$, which implies the explicit sublinear bound $|z|\le 1+\sqrt{\tfrac{n}{\log 2}}$ (Corollary \ref{cor:explicitbound}).
	
	For book graphs, we gave a description of the full limit set of domination roots, thereby providing a precise asymptotic answer to  Open Problem~\ref{problem}(iii) (Theorem \ref{thm:booklimit}). We also established parity-dependent real-root existence for even $n$ and recorded a conjectural exact classification supported by numerical evidence. Finally, for  Open Problem~\ref{problem}(iv) we proved that friendship graphs and book graphs have no nonzero integer domination roots, and combined this with known corona constructions to show that among the currently tractable families, the only nonzero integer domination root that actually occurs is $-2$.
	
	The primary restriction of this article is that the comprehensive real-root classification for book graphs remains unsolved. It remains uncertain if the modulus bound for friendship graphs is asymptotically optimal. Future directions encompass: ascertaining the precise real-root configuration of $B_n$, refining the modulus approximation for $F_n$, discovering further graph families with explicitly defined domination-root limit sets, and addressing the global integer-root problem.
	
	\section*{Declarations}
	\noindent \textbf{Data Availability:}	There is no data associated with this article.
	
	\noindent \textbf{Funding:} The authors did not receive support from any organization for the submitted work.
	
	\noindent \textbf{Conflict of interest:} The authors have no competing interests to declare that are relevant to the content of this article.
	
	\noindent\textbf{Acknowledgement:} The authors sincerely thank the anonymous reviewers for their insightful comments and constructive suggestions, which significantly improved the quality and clarity of this manuscript. We also extend our gratitude to the editorial team for their guidance and support throughout the review process.

\end{document}